\documentclass[a4paper,10pt]{article}
\usepackage{amssymb,amsmath,amsthm}
\usepackage[all,cmtip,line]{xy}
\usepackage[T1]{fontenc} 

\usepackage{tikz}
\usetikzlibrary{matrix}

\newtheorem{Theorem}{Theorem}[section]
\newtheorem{Lemma}[Theorem]{Lemma}
\newtheorem{Corollary}[Theorem]{Corollary}
\newtheorem{Proposition}[Theorem]{Proposition}

\theoremstyle{definition}
\newtheorem{Definition}[Theorem]{Definition}
\newtheorem{Remark}[Theorem]{Remark}
\newtheorem{Example}[Theorem]{Example}

\newcommand\ie{i.e.\ }
\newcommand\cf{cf.\ }

\renewcommand\o{\circ}
\newcommand \al{\alpha}

\newcommand\ze{\zeta}
\newcommand\et{\eta}
\renewcommand\th{\theta}

\newcommand\ka{\kappa}
\newcommand\la{\lambda}

\newcommand\si{\sigma}
\newcommand\ta{\tau}
\newcommand\ph{\varphi}

\newcommand\ps{\psi}
\newcommand\om{\omega}
\newcommand\Ga{\Gamma}

\newcommand\Th{\Theta}

\newcommand\Om{\Omega}

\newcommand{\half}{{\textstyle \frac{1}{2}}}

\newcommand\ham{\text {\rm ham}}

\newcommand\Diff{\text{\rm Diff}}

\newcommand\OGr{\text{\rm Gr}}
\newcommand\Gr{\text{\rm Gr}}

\newcommand\Ham{\text{\rm Ham}}
\newcommand\ex{\text{\rm ex}}
\newcommand\Quant{\text{\rm Diff}}

\newcommand\Flux{\text{\rm Flux}}

\newcommand{\oo}{\infty}

\newcommand\g{\mathfrak g}

\newcommand\n{\mathfrak n}
\newcommand\M{\mathcal{M}}
\newcommand\cP{\mathcal{P}}

\newcommand\X{\mathfrak X}

\newcommand{\pa}{\partial}

\newcommand\R{\mathbb R}
\newcommand\Z{\mathbb Z}
\newcommand\N{\mathbb N}
\newcommand\T{\mathbb T}

\begin{document}

\title
{Integrability of central extensions of the Poisson Lie algebra via prequantization} 

\author{Bas Janssens and Cornelia Vizman}

\date{May 11 2016}

\maketitle

\begin{abstract}
We present a geometric construction of central $S^1$-extensions of
the quantomorphism group of a prequantizable, compact, symplectic manifold,
{and explicitly describe}
the corresponding lattice of integrable cocycles on the Poisson Lie algebra. 
{We use this to find nontrivial central $S^1$-extensions of 
the universal cover of the group of Hamiltonian diffeomorphisms.
In the process, we obtain central $S^1$-extensions of 
Lie groups that act by exact strict contact transformations.}
\end{abstract}


\section{Introduction}

Central Lie group extensions can be obtained
by pullback of the prequantization central extension. 
The ingredients are a connected Lie group $G$ with Lie algebra $\g$, 
a connected prequantizable symplectic manifold 
$(\M,\Om)$, and
a Hamiltonian action of $G$ on $\M$.

As $\M$ is prequantizable, it has a prequantum 
$S^1$-bundle $\cP \rightarrow \M$ with connection 1-form $\Theta$,
giving rise to the \emph{quantomorphism group} 
$\Quant(\cP,\Theta)$ of connection-preserving
automorphisms of this bundle. 
Since its identity component $\Quant(\cP,\Theta)_0$ 
is a central $S^1$-extension of the Hamiltonian diffeomorphism group 
$\Ham(\M,\Omega)$,
its pullback by the Hamiltonian action $G \rightarrow \Ham(\M,\Omega)$ 
yields a central $S^1$-extension $\widehat{G}$
of $G$,
%
\begin{equation}\label{pbdiag}
\xymatrix 
{
\widehat{G} \ar[r] \ar[d]&  
\Quant(\cP,\Theta)_{0}\ar[d] 
\\
G \ar[r] &  
\Ham(\M,\Omega)\,. 
}
\end{equation}
%
If the manifold $\M$ and the Lie group $G$ are infinite dimensional, 
then this construction remains valid; the pullback $\widehat{G}$ is still a  Lie 
group, even though this may not be the case for $\Ham(\M,\Omega)$ and 
$\Quant(\cP,\Theta)$ 
(cf.\ \cite{NV03}).

We apply this construction 
in the following setting.
Suppose that $\pi:P \rightarrow M$ is a prequantum
$S^1$-bundle over a compact, symplectic manifold $(M,\omega)$ of dimension $2n$, and that $\theta$ is a connection 1-form with curvature $\omega$.
The identity component 
$G =\Quant(P,\th)_0$
of the {quantomorphism group} is then a Fr\'echet Lie group, 
with Lie algebra $\g$ isomorphic to the Poisson Lie algebra $C^{\infty}(M)$.
The infinite dimensional symplectic manifold $\M$ on which $G$ acts will be 
a connected component of the 
\emph{nonlinear Grassmannian} $\mathrm{Gr}_{2n-1}(P)$
of codimension two closed, oriented,
embedded submanifolds of $P$. 
This Fr\'echet manifold
is prequantizable by \cite{I96,HV04}, and the natural action of $G$ on $\M$ is Hamiltonian.
In this way, we obtain central $S^1$-extensions $\widehat{G} \rightarrow G$
of the identity component $G = \Quant(P,\th)_0$ of the quantomorphism group.

For any (Fr\'echet) Lie group $G$, the central extensions of
$G$ by $S^1$ play 
a pivotal role in the theory of \emph{projective unitary $G$-representations}. 
Every such 
representation 
gives rise to a central $S^1$-extension $\widehat{G}$, together with a \emph{linear}
unitary $\widehat{G}$-representation
\cite{PS86, TL99, JN15}.
Passing to the infinitesimal level, one obtains information on the 
projective $G$-representations from the 
(often more accessible) linear representation theory
of the corresponding central Lie algebra extensions $\widehat{\g}$.

In the passage to the infinitesimal level, however, one important piece of information is lost: 
not every Lie algebra extension $\widehat{\g} \rightarrow \g$
integrates to a group extension $\widehat{G} \rightarrow G$.
The ones that do, determine a lattice $\Lambda \subseteq H^2(\g,\R)$
in the continuous second Lie algebra cohomology of $\g$, 
called the \emph{lattice of integrable classes}.

In the context of quantomorphism groups, 
the continuous second Lie algebra cohomology 
of the Poisson Lie algebra $C^{\infty}(M)$
has been explicitly determined: in \cite{JV}, we proved that 
\[
	H^2(C^\oo(M),\R) \simeq H^1(M,\R)\,.
\]
To the best of our knowledge, it remains an open problem to determine the 
full lattice $\Lambda \subseteq H^1(M,\R)$ of integrable classes; it appears 
that  
the period homomorphism 
governing integrability (cf.\ \cite[Thm.\ 7.9]{N02}) is not easy to calculate 
in the setting of quantomorphism groups.

In the present paper, we contribute towards a solution by 
explicitly determining
the sublattice $\Lambda_0 \subseteq \Lambda$
corresponding to the group extensions $\widehat{G} \rightarrow G$
described above.
We find that 
\[
	\Lambda_0 = \frac{n+1}{2\pi \mathrm{vol}(M)}\pi_! (H^2(P,\R)_{\Z})\,,
\]
where $H^2(P,\R)_{\Z}$ is the lattice of integral classes in de Rham cohomology,
and $\pi_{!}$ denotes  fiber integration.
This formula is easily evaluated in concrete situations.
If $M$ is a compact surface,  then 
$\Lambda_0 \subseteq H^1(M,\R)$ is of full rank.
On the other extreme, we find that $\Lambda_0 = \{0\}$ 
if $M$ is a compact K\"ahler manifold of dimension $2n \geq 4$.
Intermediate behavior is displayed by \emph{nilmanifolds}. 
Thurston's nilmanifold $M^4$, for example, affords a lattice $\Lambda_0$
that is of rank $1$ in the 3-dimensional vector space
$H^1(M^4,\R)$.

%


We expect that in the representation theory of the Poisson Lie algebra $C^{\infty}(M)$,
the lattice $\Lambda$  
of integrable classes will play the same role as the 
\emph{integral level condition} in the
representation theory of loop algebras and affine Kac-Moody 
algebras \cite[\S 12]{Ka90}. From a differential geometric point of view, 
integrality of the level
for a loop algebra $\g$ (possibly twisted, 
over a simple Lie algebra) is precisely the condition 
that the induced class in $H^2(\g,\R) \simeq H^1(S^1,\R)$
corresponds to a Lie group extension,
cf.\ \cite[\S 4]{PS86}.

\section{Prequantization central extension}

Let $G$ be a connected Lie group with Lie algebra $\g$,  and let 
$(\M, \Omega)$ be a connected 
prequantizable symplectic manifold with a Hamiltonian $G$-action.
Both $G$ and $\M$ are allowed to be 
infinite dimensional manifolds, modeled on locally convex spaces.
Let $\pi\colon \cP\to \M$ be a {\it prequantum bundle},
\ie a principal $S^1$-bundle with principal connection 
1-form $\Th$ and curvature $\Om$. In particular,
the identity $\pi^*\Om=d\Th$ holds.

\subsection{The Kostant-Souriau extension}

The prequantum bundle $\cP \rightarrow \M$ gives rise to the
{\it prequantization central extension} \cite{Ko70, S70} 
\begin{equation}\label{kostant}
S^1\to\Quant(\cP,\Th)_0\to\Ham(\M,\Om),
\end{equation}
where $\Ham(\M,\Om)$ is the group of Hamiltonian diffeomorphisms, and 
$\Quant (\cP,\Th)_0$ 
is the identity component of the \emph{quantomorphism group}
\[
	\Quant(\cP,\Th)=\{\ph\in\Diff(\cP):\ph^*\Th=\Th\}\,.
\]
Note that $\ph^*\Th = \Th$ implies $\ph_*E = E$, where $E \in \X(P)$ is the infinitesimal generator 
of the $S^1$-action. In particular, every quantomorphism is a bundle automorphism.

The infinitesimal counterpart of \eqref{kostant} is the
central
extension
\begin{equation}\label{extquant}
\R \rightarrow \X(\cP,\Th) \rightarrow \X_{\ham}(\M,\Omega)
\end{equation}
of the Lie algebra of \emph{Hamiltonian} vector fields
\[
	\X_{\ham}(\M,\Omega) := \{ X_{f} \in \X(\M) : i_{X_{f}}\Omega = -df \},
\]
namely the \emph{quantomorphism Lie algebra}
\[
\X(\cP,\Th) := \{X\in\X(\cP) \colon L_X\Th=0\}\,.
\]
It is isomorphic to the Poisson Lie algebra $C^{\infty}(\M)$
via the Lie algebra isomorphism
\begin{equation}\label{iso}
\zeta \colon C^{\infty}(\M) \rightarrow \X(\cP,\Th),
\quad\zeta_{f}:=X_{f}^{\rm hor}+(\pi^*f)E,
\end{equation}
where $Y^{\rm hor}$ denotes the horizontal lift of the vector field $Y\in\X(\M)$.
The central extension \eqref{extquant} can thus be identified with the \emph{Kostant-Souriau} extension
\begin{equation}\label{poissonrijtje}
\R \rightarrow C^{\infty}(\M) \rightarrow \X_{\ham}(\M,\Omega)\,,
\end{equation}
induced by the map $f \mapsto X_{f}$.

The central extensions of a locally convex Lie algebra $\g$ are classified by 
its \emph{continuous second Lie algebra cohomology} $H^2(\g,\R)$, cf.\ e.g.\ \cite[\S 2.3]{JV}.
This is the cohomology of the cochain complex $C^{n}(\g,\R)$
of continuous, alternating, $n$-linear maps $\g^n \rightarrow \R$, with differential
$\delta \colon C^{n}(\g,\R) \rightarrow C^{n+1}(\g,\R)$ defined by
\begin{equation}\label{CE}
\delta \psi(x_0,\ldots,x_{n}):= \sum_{0\leq i<j\leq n}
(-1)^{i+j} \psi([x_i,x_j],x_0,\ldots,\widehat{x}_i, \ldots, 
\widehat{x}_j, \ldots, x_n)\,.
\end{equation}

A continuous linear splitting $s \colon \g \rightarrow \widehat{\g}$ of the central extension 
$\R \rightarrow \widehat{\g} \rightarrow \g$
gives rise to the 
2-cocycle $\psi(X,Y) := [s(X),s(Y)] - s([X,Y])$.
Given a point $x_0 \in \M$, we split \eqref{poissonrijtje} by mapping 
$X_{f} \in \X_{\ham}(\M,\Omega)$ to the unique Hamiltonian function $f^{x_0}$
vanishing on $x_0$.
The splitting
\begin{equation}\label{splitting}
s_{x_0} \colon X_{f} \mapsto f^{x_0}
\end{equation}
yields the \emph{Kostant-Souriau cocycle} $\psi_{KS}$ on $ \X_{\ham}(\M)$
that is given by
\begin{equation}\label{KS}
\ps_{KS}(X_f,X_g)=\{f,g\}(x_0) = \Om(X_f,X_g)(x_0)\,.
\end{equation} 

\subsection{Group extensions from Hamiltonian actions}

If $\M$ is a compact (hence finite dimensional) manifold,
then $\Ham(\M,\Om)$
and $\Quant(\cP,\Theta)$ are both Fr\'echet Lie groups, with Lie algebras
$\X_{\ham}(\M,\Omega)$ and $\X(\cP,\Th)$, respectively, see
 \cite[\S 3]{RS}, 
{\cite[\S VIII.4]{Om74}}.
Even though this need no longer be the case if $\M$ is infinite dimensional, 
it is still true that the pullback $\widehat{G}$ of \eqref{kostant} under the 
Hamiltonian action of a Lie group $G$ on $\M$ has a smooth Lie group structure.
\begin{Theorem}{\rm \cite[Thm 3.4]{NV03}}\label{pullb}
Let $(\M,\Om)$ be
a prequantizable, symplectic manifold
with a Hamiltonian action of
a connected Lie group $G$.
Then 
the pullback of the prequantization central extension \eqref{kostant} 
by the action $G\to\Ham(\M,\Om)$ 
provides a  central Lie group extension 
\begin{equation}
S^1 \rightarrow \widehat{G} \rightarrow G\,.
\end{equation}
\end{Theorem}
The derived Lie algebra extension 
$\R \rightarrow \widehat{\g} \rightarrow \g$
is given by the pullback 
of the Kostant-Souriau extension \eqref{poissonrijtje}
along the infinitesimal action $\g \rightarrow \X_{\ham}(\M,\Omega)$.
The linear splitting \eqref{splitting}
therefore induces a linear 
splitting of $\widehat{\g} \rightarrow \g$, and 
the corresponding 2-cocycle $\si$ on $\g$ 
is the pullback by $\g \rightarrow \X_{\ham}(\M,\Omega)$ of the Kostant-Souriau cocycle. 
It is given explicitly by
\begin{equation}\label{splittingCE}
\si(\xi,\et)= \Om(\xi_\M,\et_\M)(x_0)\,,
\end{equation}
where  $\xi_\M$ denotes the fundamental vector field on $\M$ for $\xi\in\g$.


\section{Exact volume preserving diffeomorphisms}

In order to obtain central extensions of the Lie group $G = \Diff_{\ex}(M,\nu)$ of exact volume 
preserving diffeomorphisms of a compact manifold $M$
endowed with volume form $\nu$, we consider its Hamiltonian action
on the non-linear Grassmannian of codimension 2 
embedded submanifolds of $M$.

\subsection{Non-linear Grassmannians}

Let $M$ be a closed, connected manifold of  dimension $m$.
The {\it non-linear Grassmannian} $\Gr_{k}(M)$ consists of $k$-dimensional, closed, oriented, embedded 
submanifolds $N \subseteq M$.  It is a Fr\'echet manifold in a natural way, cf.\ 
\cite{KM97, GV14}.
The tangent space of $\OGr_k(M)$ at $N$ can
naturally be identified with the space of smooth sections of the normal
bundle $TN^\perp:=(TM|_N)/TN$.

For every $r\geq 0$, the transgression map $\tau \colon \Omega^{k+r}(M) \rightarrow \Omega^{r}(\Gr_{k}(M))$ is defined by
$$
(\tau{\alpha})_{N}([Y_1],\ldots,[Y_r]) := \int_{N}i_{Y_r}\ldots i_{Y_{1}}(\alpha|_{N}).
$$
Here all $[Y_j]$ are tangent vectors at $N\in\OGr_k(M)$, \ie sections
of $TN^\perp$. The expression above is independent of the vector fields $Y_j$ on $M$ along $N$ chosen to represent $[Y_j]$.

The natural group action of $\mathrm{Diff}(M)$ on $\Gr_{k}(M)$, defined by 
$(\ph , N) \mapsto \ph(N)$, is smooth, since it
descends from the action of $\Diff(M)$ on the manifold
of embeddings into $M$ defined by composition
$(\ph , f) \mapsto \ph \o f$. It differentiates to
the Lie algebra action  $\X(M) \rightarrow \X(\Gr_{k}(M))$
given by $X \mapsto \ta_X$ with $\ta_X(N) = [X|_{N}]$.
The transgression enjoys the following functorial properties:
\begin{align}\label{fun}
&d\circ \tau = \tau \circ d,\quad\quad\quad \ph^* \circ \tau = \tau \circ \ph^*, \\
&i_{\ta_{X}}\circ \tau = \tau\circ i_{X},\quad L_{\ta_X} \circ \tau = \tau\circ L_{X}.\nonumber
\end{align}


\begin{Theorem}{\rm \cite[\S 25.3]{I96} \cite[Thm.\ 1]{HV04}}\label{taual} 
Let $\al\in\Om^{k+2}(M)$ be a closed differential form with integral cohomology class. Then the non-linear Grassmannian $\OGr_{k}(M)$ endowed with the closed 2-form $\Om=\tau\al$ is prequantizable,
\ie there exist an $S^1$-bundle
$\mathcal P\to\OGr_{k}(M)$ with
connection form $\Th\in\Omega^1(\mathcal P)$ and curvature~$\Om$.
\end{Theorem}


\subsection{Lichnerowicz central extensions}

Let $\nu \in \Omega^{m}(M)$ be a volume form on $M$, normalized so that $\mathrm{vol}_{\nu}(M) = 1$.
It induces a symplectic form $\Om=\tau\nu$
on the codimension two non-linear Grassmannian $\OGr_{m-2}(M)$ \cite{I96}.
This is the higher dimensional version of the natural symplectic form on the space of knots in $\R^3$ \cite{MW83}.
The natural action of the group $\Diff(M,\nu)$ of volume preserving diffeomorphisms on $\OGr_{m-2}(M)$ is symplectic, as
$\ph^*\ta\nu=\ta\ph^*\nu=\ta\nu$ for all $\ph\in\Diff(M,\nu)$
by \eqref{fun}. To get a Hamiltonian action, we have to restrict to the subgroup $\Diff_{\ex}(M,\nu)$ of \emph{exact volume preserving diffeomorphisms}. 

Its Lie algebra ${\X}_{\ex}(M,\nu)$ of \emph{exact divergence free vector fields}
is the kernel of the infinitesimal flux homomorphism, defined on the Lie algebra 
$\X(M,\nu)$ of divergence free vector fields by
\begin{equation}\label{flu}
\X(M,\nu)\to H^{m-1}(M,\R),\quad X\mapsto[i_X\nu].
\end{equation}
We denote by $X_\al$ the exact divergence free vector field with potential 
$\al\in\Om^{m-2}(M)$, \ie $i_{X_\al}\nu=d\al$.

The Lie algebra homomorphism \eqref{flu} is integrated by 
\emph{Thurston's flux homomorphism}.
On the universal cover of the identity component of the group
of volume preserving diffeomorphisms, we define
\begin{equation}\label{split1}
\widetilde{\Flux}:\widetilde{\Diff}(M,\nu)_0\to H^{m-1}(M,\R) \quad \text{by} \quad
\widetilde{\Flux}([\ph_{t}])=\int_0^1[i_{X_t}\nu]dt\,,
\end{equation} 
where $\ph_t$ is a volume preserving isotopy from the identity to $\ph$, 
and $X_t$ is the time dependent vector field such that $\frac{d}{dt}\ph_t=X_t\o\ph_t$.
By \cite[Thm.\ 3.1.1]{B}, this is a well defined group homomorphism.

For any codimension one submanifold $N \subset M$,
the integral $\int_{N} \widetilde{\Flux}([\ph_{t}])$ 
is the volume swept out by  $N$ under $\ph_{t}$. 
Therefore, the monodromy
subgroup $\Gamma := \widetilde{\Flux}(\pi_1(\Diff(M,\nu)_0))$  of $H^{m-1}(M,\R)$
is discrete. It follows that 
equation \eqref{split1} factors through a Lie group homomorphism
\begin{equation}\label{split}
\Flux:\Diff(M,\nu)_0\to H^{m-1}(M,\R)/\Ga,\quad
\Flux(\ph)=\int_0^1[i_{X_t}\nu]dt\mod\Ga\,.
\end{equation} 
The group of exact volume 
preserving diffeomorphisms $\Diff_{\ex}(M,\nu)$ is now defined as the kernel of the 
Flux homomorphism;  
it is a Lie group with Lie algebra ${\X}_{\ex}(M,\nu)$ \cite{B, KM97}.


Since $\mathrm{vol}_{\nu}(M)$ is normalized to $1$, the cohomology class $[k\nu]\in H^m(M,\R)$
is integral for every $k\in \Z$.
By Theorem \ref{taual}, this implies that
the manifold $\OGr_{m-2}(M)$ with symplectic form $\Om=k\, \tau\nu$ is prequantizable.
The natural action of $\Diff_{\ex}(M,\nu)$
on $\OGr_{m-2}(M)$ is Hamiltonian,
as $i_{\ta_{X_\al}}\ta\nu=\ta i_{X_\al}\nu=\ta d\al=d\ta\al$
for all $X_\al\in\X_{\ex}(M,\nu)$ by \eqref{fun}.
Now we can apply Theorem \ref{pullb} to this Hamiltonian action  
on a connected component $\mathcal M$ of $\OGr_{m-2}(M)$.
This yields the central Lie group extension 
\begin{equation}\label{extdiffexgroup}
S^1 \rightarrow \widehat{\Diff}_{\ex}(M,\nu) \rightarrow \Diff_{\ex}(M,\nu)
\end{equation}
of the group of exact 
volume preserving diffeomorphisms.

To obtain the corresponding Lie algebra 2-cocycle, we fix a point $Q \in \M$, that is, a codimension two submanifold $Q \subset M$ in the connected component $\M$
of the nonlinear Grassmannian.
By \eqref{splittingCE}, the Lie algebra extension of $\X_{\ex}(M,\nu)$ corresponding to
\eqref{extdiffexgroup} for $k = 1$, is described by the 
Lie algebra 2-cocycle  
\begin{equation}\label{lich}
\la^{\nu}_{Q}(X,Y)=(\tau \nu)_{Q}(\ta_X,\ta_Y)
= \int_{Q}i_Yi_X\nu
\end{equation}
on $\X_{\ex}(M,\nu)$, which we call the \emph{singular Lichnerowicz cocycle}.
If the class $[k\nu]$ is used to construct the extension, then the corresponding 
2-cocycle is $k\la_{Q}^{\nu}$.

\begin{Theorem}\label{lint}{\rm\cite[\S 25.5]{I96}\cite[Thm.\ 2]{HV04}}
Let $\nu$ be a volume form on $M$ with $\mathrm{vol}_{\nu}(M)=1$
and $Q$ a codimension two embedded submanifold of $M$. Then the
Lie algebra extensions defined by integral multiples of the
cocycle $\la^{\nu}_{Q}$ of equation \eqref{lich}
integrate to central Lie group
extensions of the group of exact volume preserving diffeomorphisms $\Diff_{\ex}(M,\nu)$.
\end{Theorem}

Recall that two classes $[Q] \in H_{m-k}(M,\R)$ and 
$[\alpha] \in H^{k}(M,\R)$ are called \emph{Poincar\'e dual} if 
$\int_{Q} \gamma = \int_{M} \eta \wedge \gamma$ for all closed 
$\gamma \in \Omega^{m-k}(M)$.
If $[\et]\in H^2(M,\R)$ is Poincar\'e dual to 
$[Q]\in H_{m-2}(M,\R)$, then by \cite[Prop.\ 2]{Vi}
the cocycle $\la^{\nu}_{Q}$ is cohomologous to the {\it Lichnerowicz cocycle} \cite{Li74}
\begin{equation}\label{lichcoc}
\la^{\nu}_\et(X,Y)=\int_M\et(X, Y)\nu=\int_M\et\wedge i_Yi_X\nu.
\end{equation}
{If $\dim M\ge 3$,} the map  $[\et]\mapsto[\la^\nu_\et]$ is believed to be an isomorphism between 
$H^2(M,\R)$ and the second Lie algebra cohomology
group $H^2(\X_{\ex}(M,\nu),\R)$, see \cite[\S 10]{R95} for 
the outline of a proof.

\begin{Remark}\label{pd}
If $[\eta]$ is Poincar\'e dual to $[Q]$ with $Q\in\Gr_{m-2}(M)$, then in particular, it is 
an integral cohomology class. Conversely,
every integral cohomology class $[\et]\in H^2(M,\R)$ is the Poincar\'e dual
of a closed submanifold of codimension two in $M$;
it can be obtained (cf.\ \cite[Prop.\ 12.8]{BT82}) as the zero set of a section transversal to the zero section 
in a rank two vector bundle with Euler class $[\et]$. 
\end{Remark}

We infer that the Lichnerowicz cocycle \eqref{lichcoc} gives rise to an 
\emph{integrable} 
Lie algebra extension if $[\eta] \in H^2(M,\R)$ is an \emph{integral} class 
in de Rham cohomology
(in the sense that on integral singular 2-cycles, it evaluates to an integer). 
We denote by  $H^2(M,\R)_{\Z}$ 
the space of integral de Rham classes.
It follows that the image of $H^2(M,\R)_{\Z}$
by the map $[\et]\mapsto [\la_\et^{\nu}]$,
which coincides with the image of $H_{n-2}(M,\Z)$ by the map $[Q]\mapsto[\la_Q^{\nu}]$,
lies in the lattice of integrable classes in $H^2(\X_{\ex}(M,\nu),\R)$.


\section{Strict contactomorphisms}
Let $P$ be a compact manifold of dimension $2n+1$, 
equipped with a contact 1-form $\th$. 
The group $\Diff(P,\th)$ of strict contactomorphisms 
is a subgroup of the volume preserving diffeomorphism group $\Diff(P,\mu)$,
where the volume form $\mu$ is a constant multiple of $\th\wedge(d\th)^n$.
The group of {\it exact strict contactomorphisms} is defined as
\begin{equation}
\Diff_{\ex}(P,\th) := \Diff(P,\th)_{0} \cap \Diff_{\ex}(P,\mu)\,.
\end{equation}
We use Theorem \ref{lint} to investigate central extensions of 
locally convex Lie groups $G$
that act on $P$ by exact strict contactomorphisms.

The motivating example is the case where $P \rightarrow M$ is a 
prequantum bundle over a compact symplectic manifold 
$(M,\omega)$; in this case $\Diff(P,\th)$ is the quantomorphism group, while 
$\Diff_{\ex}(P,\th)$ 
coincides with the identity component $\Quant(P,\th)_0$ of the quantomorphism group.
Since the latter is a locally convex Lie group (it is even an ILH-Lie group by 
\cite[VIII.4]{Om74}), we thus obtain central 
Lie group extensions of $G = \Quant(P,\th)_0$ by $S^1$.

\subsection{Strict contactomorphisms}
The contact form $\th$ gives rise to the volume form $\th \wedge (d\th)^n$. We will use 
the following two normalizations of this form: 
\begin{equation}\label{normalisation}
\mu := \frac{1}{(n+1)!}\,\th \wedge (d\th)^{n}\,, \quad \text{and}
\quad
\nu :=  \frac{1}{\mathrm{vol}_{\mu}(P)} \,\mu\,.
\end{equation}
The \emph{Reeb vector field} $E \in \X(P)$ is uniquely determined by $i_{E}\th = 1$ and ${i_{E}d\th = 0}$. 
We define the \emph{strict contactomorphism group} 
\[
\mathrm{Diff}(P,\th) := \{\ph \in \mathrm{Diff}(P)\,;\, \ph^* \th = \th\}\subset\Diff(P,\mu)
\]
and the \emph{Lie algebra of strict contact vector fields}
\[
\X(P,\th) = \{X \in \X(P)\,;\, L_{X} \th = 0\}\subset\X(P,\mu)\,.
\]
For every isotopy $\ph_t\in\Diff(P)$, starting at the identity and 
corresponding to the time dependent vector field $X_t\in\X(P)$,
we have $\ph_t\in\Diff(P,\th)$ for all $t$ if and only if 
the vector field $X_t\in\X(P,\th)$ for all $t$. 

A Hamiltonian function $f \in C^{\infty}(P)^{E} =\{f\in C^\oo(P):L_Ef=0\}$
defines a unique  strict contact vector field
$\zeta_{f} \in \X(P,\th)$ by 
\[i_{\zeta_{f}} \th = f, \quad i_{\zeta_{f}}d\th = -df\,.\]
The corresponding map 
\begin{equation}\label{zet}
\ze:C^{\infty}(P)^{E} \rightarrow \X(P,\th)
\end{equation}
is an isomorphism of Fr\'echet Lie algebras, if $C^\oo(P)^E$ is equipped with the 
Lie bracket 
\begin{equation}\label{lb}
\{f,g\} = d\th(\zeta_{f},\zeta_{g})=L_{\ze_f}g\,.
\end{equation}

\begin{Proposition}\label{trekterug}
Let $Q$ be a codimension two submanifold of the contact manifold $(P,\th)$.
Then the pullback to $C^\oo(P)^E$
of the singular Lichnerowicz cocycle $\la^{\mu}_{Q}$
on $\X(P,\mu)$ by the map $\ze$ in \eqref{zet}
is given by 
\[
(\zeta^* \la^{\mu}_{Q})(f,g) = \si_{Q}(f,g) + \frac{1}{n+1} \delta \rho_{Q}(f,g)\,, 
\]
where 
the 2-cocycle $\si_{Q}$ and the 1-cochain $\rho_{Q}$ on $C^{\infty}(P)^{E}$ are given by

\begin{eqnarray}\label{sg2}
\si_Q(f,g) & :=  & \phantom{-}\! \int_Q gdf\wedge(d\th)^{n-1}/(n-1)!\,,\\
\rho_{Q}(h) & := & -\int_{Q}h\th\wedge(d\th)^{n-1}/(n-1)!\,.
\end{eqnarray}
\end{Proposition}
\begin{proof}
We calculate $\zeta^*\la_{Q}^{\mu}(f,g) = \la^{\mu}_{Q}(\zeta_{f}, \zeta_{g}) = 
\int_{Q} i_{\zeta_{g}}i_{\zeta_{f}}(\th\wedge (d\th)^{n})/(n+1)!$. 
First, note that
\begin{eqnarray}
i_{\zeta_{g}}i_{\zeta_{f}}(\th\wedge (d\th)^{n}/n!) &=&
i_{\zeta_{g}}(f(d\th)^n/n! + \th\wedge df \wedge (d\th)^{n-1}/(n-1)!) \nonumber\\
&=& -(fdg-gdf)\wedge(d\th)^{n-1}/(n-1)!  \label{pullLichnerowicz}\\
&& + \{f,g\} \,\th \wedge (d\th)^{n-1}/(n-1)!\nonumber \\
& &-\th\wedge df \wedge dg \wedge (d\th)^{n-2}/(n-2)!\,.\nonumber
\end{eqnarray}
Expanding $d\left(\th\wedge(fdg-gdf) \wedge (d\th)^{n-2}\right)$, we obtain
\begin{eqnarray*}
\th\wedge df \wedge dg \wedge (d\th)^{n-2}/(n-2)!
&=&
\phantom{+}\half(n-1) (fdg-gdf)\wedge(d\th)^{n-1}/(n-1)!\\
& & - \half d\left(\th\wedge(fdg-gdf) \wedge (d\th)^{n-2}/(n-2)!\right).
\end{eqnarray*}
Inserting this into \eqref{pullLichnerowicz} yields
\begin{eqnarray}
i_{\zeta_{g}}i_{\zeta_{f}}(\th\wedge (d\th)^{n}/n!) 
&=& - (n+1)\half (fdg-gdf)\wedge(d\th)^{n-1}/(n-1)!  \label{pullLichnerowicz2}\\
&& +  \{f,g\}\th\wedge (d\th)^{n-1}/(n-1)!\nonumber \\
& &+ \half d\left(\th\wedge(fdg-gdf) \wedge (d\th)^{n-2}/(n-2)!\right)\,.\nonumber
\end{eqnarray}
Since the value of $\la_{Q}^{\mu}(\zeta_{f}, \zeta_{g})$ is obtained by  
integrating the above expression over $Q$ and dividing by $n+1$,
the last term vanishes ($Q$ is closed),
the middle term yields a multiple of the coboundary $\delta \rho_{Q}(f,g) = \rho_{Q}(-\{f,g\})$, and 
the first term yields the cocycle $\si_{Q}(f,g)$.
\end{proof}
In particular the classes $\ze^*[\la_Q^\mu]$ and $[\si_Q]$ in 
$H^2(C^\oo(P)^E,\R)$ coincide.

\begin{Remark}[\bf Regular contact manifolds]\label{BWremark}
For us, the motivating example is the total space $(P,\th)$ of a prequantum $S^1$-bundle 
$\pi \colon P \rightarrow M$ 
 over a compact, symplectic manifold $(M,\omega)$.
 These are called \emph{regular} or 
 \emph{Boothby-Wang} contact manifolds \cite{BW}. 
The top form 
\[\mu = \frac{1}{(n+1)!}\th \wedge (d\th)^{n} = \frac{1}{(n+1)!}\th \wedge \pi^*\omega^n\]
is a volume form, and the Reeb vector field 
$E\in\X(P)$ coincides with the infinitesimal generator 
 of the principal $S^1$-action.
The group  of strict contactomorphisms coincides with
the quantomorphism group, hence it is a
Fr\'echet Lie group \cite{Om74,RS}.

The pullback by $\pi $ is an isomorphism
between the Poisson Lie algebra $C^\oo(M)$ and 
the Lie algebra $C^\oo(P)^E$ with Lie bracket \eqref{lb}.
Under the identification $C^{\infty}(M) \simeq C^{\infty}(P)^{E}$,
the isomorphisms \eqref{iso} and \eqref{zet}, both denoted by $\ze$, coincide.
Moreover, the cocycle $\si_Q$ in \eqref{sg2} can be identified with the cocycle $\psi_{\pi_*Q}$
on the Poisson Lie algebra, determined by 
the singular $2n-1$ cycle $C=\pi_*Q$ in $M$ by the formula \cite{JV}:
\[
\ps_C(f,g)=\int_Cgdf\wedge\om^{n-1}/(n-1)!,\quad f,g\in  C^\oo(M)\,.
\]
More details will be given in Section 5.
\end{Remark}

\subsection{Exact strict contactomorphisms}

Suppose that $G$ is a locally convex Lie group that acts 
smoothly {and effectively} on $P$
by exact strict contact transformations.
{Its Lie algebra $\g$ is then a subalgebra of 
$\X(P,\th)\simeq C^\oo(P)^E$.
We investigate the integrability of
the pullback $\iota^*[\si_Q] \in H^2(\g,\R)$
along the inclusion $\iota \colon \g \hookrightarrow C^{\infty}(P)^{E}$, where
$\si_Q$ is the cocycle
$\si_Q(f,g) =  \int_Q gdf\wedge(d\th)^{n-1}/(n-1)!$ of  \eqref{sg2}.}
For general contact manifolds, we have to impose the condition
of exactness so that we can make use of Theorem \ref{lint}.
However,  
in the important special case of regular contact manifolds,
we will show that
the exactness condition is automatically satisfied.

Analogous to the group $\Diff_{\ex}(P,\th) := \Diff(P,\th)_{0} \cap \Diff_{\ex}(P,\mu)$
of exact strict contact transformations, we define the Lie algebra of
{\it exact strict contact vector fields} by
$$\X_{\ex}(P,\th):=\X(P,\th) \cap \X_{\ex}(P,\mu)\,.$$

For every isotopy $\ph_t\in\Diff(P)$, starting at the identity
and determined by the time dependent vector field $X_t\in\X(P)$,
we have $\ph_t\in \Diff_{\ex}(P,\th)$ for all $t$ if and only if 
$X_t\in\X_{\ex}(P,\th)$ for all $t$.


\begin{Lemma}
The function space
\[
C^{\infty}_{0}(P)^{E} := \{f \in C^{\infty}(P)^{E}\,;\,  f(d \th)^n \in d\Omega^{2n-1}(P)\}
\]
is a Lie subalgebra of $C^\oo(P)^E$ of finite codimension
$\le\dim H^{2n}(P,\R)$,
isomorphic under $f\mapsto\ze_f$  to 
the Lie algebra $\X_{\ex}(P,\th)$
of  exact strict contact vector fields.
\end{Lemma}

\begin{proof}
As $f(d\th)^n$ is closed for all $f \in C^{\infty}(P)^{E}$, we can define the
linear map 
\begin{equation}\label{liho}
C^{\infty}(P)^{E} \rightarrow H^{2n}(P,\R), 
\quad f \mapsto \frac{1}{n!}[f (d\th)^n]
\end{equation}
with kernel  $C^{\infty}_{0}(P)^{E}$.
This is a Lie algebra homomorphism, as 
$\{f,g\}(d\th)^n =ndf\wedge dg\wedge (d\th)^{n-1}$ is exact. 
It coincides, under the identification \eqref{zet}, with the flux homomorphism \eqref{flu} 
for the volume form $\mu$
restricted to  $\X(P,\th)\simeq C^{\infty}(P)^{E}$, as
\begin{eqnarray*}
(n+1)! [i_{\ze_f}\mu]&= &[i_{\ze_f}(\th\wedge (d\th)^n)]\\
&=& [f(d\th)^n- n (df)\wedge\th\wedge(d\th)^{n-1}]\\
&=& [f (d\th)^n- n d(f\th\wedge(d\th)^{n-1}) + n f (d\th)^n]\\
&=& (n+1)[f (d\th)^n]\,.
\end{eqnarray*}
It follows that the kernel $C_0^\oo(P)^E$ of \eqref{liho} 
is identified under $\ze$ with the exact strict contact vector fields 
$\X(P,\th)\cap\X_{\ex}(P,\mu)=\X_{\ex}(P,\th)$.
\end{proof}

\begin{Proposition}\label{LAsame}
If the contact manifold $(P,\th)$ is regular, i.e.\ the total space of a prequantum bundle $\pi:P\to M$,  
then the Lie algebras of strict contact and exact strict contact vector fields coincide:
$\X(P,\th)=\X_{\ex}(P,\th)$.
Moreover,  the group of exact strict contact diffeomorphisms
is precisely the connected component of the 
quantomorphism group: 
\[ \Diff_{\ex}(P,\th)=\Diff(P,\th)_{0} \,.\]
\end{Proposition}

\begin{proof}
Any $f\in C^\oo(P)^E$ 
is of the form $\pi^*\bar f$ for a smooth function $\bar f$ on the compact symplectic manifold $M$.
If we write $\bar f =\bar f_0 + c$ with $\int_{M} \bar{f}_{0} \omega^n = 0$, then
$\bar f_0 \omega^n = d\gamma$ is exact, so that also
 $f (d\th)^n=c(d\th)^n+\pi^*(\bar f_0\om^n) = c d(\th \wedge (d\th)^{n-1}) + d\pi^*\gamma$ is exact.
Hence $C_0^\oo(P)^E=C^\oo(P)^E\simeq C^\oo(M)$ and the conclusion follows.
\end{proof}

The following example shows that for contact manifolds that are not regular, 
the Lie algebra 
$C_{0}^{\infty}(P)^{E}$ can be strictly smaller than $C^{\infty}(P)^{E}$.

\begin{Example} 
An example of a non-regular contact form on the 3-torus $P=\T^3$
is $\th=\cos z dx+\sin z dy$.
The orbits of the Reeb vector field $E=\cos z \pa_x+\sin z \pa_y$
determine constant slope foliations 
on each 2-torus of constant $z$.
We show that $C_0^\oo(P)^E \simeq \X_{\ex}(P,\th)$
has codimension two in $C^\oo(P)^E\simeq\X(P,\th)$.

We use the inclusion $\X(P,\th)\subset\X(P,\mu)$.
Any divergence free vector field $X\in\X(P,\mu)$ is the sum $X = X_0 + X_{\alpha}$
of an exact divergence free vector field
$X_\al$
with potential 1-form $\al = Adx + Bdy + Cdz$ and 
a constant vector field $X_0 = a \partial_x + b \partial_y + c \partial_z$.
With volume form $\mu =\half\th\wedge d\th= -\half dx\wedge dy\wedge dz$, we have 
$X_\al=2(B_z-C_y)\pa_x+2(C_x-A_z)\pa_y+2(A_y-B_x)\pa_z$.
The vector field $X$ is  strict contact if $L_{X}\th = 0$, 
which amounts to
\begin{eqnarray*}
\sin z(C_{xx}-A_{xz}-A_y+B_x-c)+\cos z(B_{xz}-C_{xy})&=&0\\
\cos z(B_{yz}-C_{yy}+A_y-B_x+c)+\sin z(C_{xy}-A_{yz})&=&0\\
\cos z (B_{zz} - C_{yz}) + \sin z (C_{xz} - A_{zz}) &=& 0.
\end{eqnarray*}
Thus $a, b \in \R$ are arbitrary, $c=0$ (as can be seen by integrating the above equations 
over $x$ and $y$), and $X_\al$ is {a strict contact vector field.} 
We find {that} $\X(P,\th)$ is isomorphic to the semidirect product
$\mathrm{Span}\{\pa_x, \pa_y\} \ltimes \X_{\rm ex}(P,\th)$, and {that}
the flux homomorphism \eqref{flu} restricted to $\X(P,\th)$
has 2-dimensional  image generated by $[dy\wedge dz], [dx\wedge dz]\in H^2(P,\R)$.
\end{Example}


We apply Theorem \ref{lint} 
to the contact manifold $P$ with integral volume form $k\nu=\frac{k}{\mathrm{vol}_{\mu}(P)}\mu$ for $k \in \Z$
and we obtain the following central result.

\begin{Theorem}\label{centralresult}
Let $G$ be a Lie group acting {smoothly} on $(P,\th)$ by exact strict contact transformations.
Then the restriction to $\g$ of the class 
$\frac{k}{\mathrm{vol}_{\mu}(P)} [\si_{Q}]$
is integrable to a central Lie group extension of $G$.
\end{Theorem}

\begin{proof}
Since $G \subseteq \Diff_{\ex}(P,\th)$ and $\Diff_{\ex}(P,\th) \subseteq \Diff_{\ex}(P,k\nu)$,
the action of $G$ on the connected component $\M$ of $\mathrm{Gr}_{2n-1}(P)$
is Hamiltonian.
Theorem~\ref{pullb} then yields a 
Lie group extension 
\[S^1 \rightarrow \widehat{G} \rightarrow G\,.\]
By Theorem~\ref{lint}, the corresponding class in $H^2(\g,\R)$ 
is the pullback along the inclusion $\iota \colon \g \hookrightarrow \X_{\mathrm{ex}}(M,\nu)$ 
of 
$[\lambda^{k\nu}_{Q}] = \frac{k}{\mathrm{vol}_{\mu}(P)}[\lambda^{\mu}_{Q}]$. 
By Proposition~\ref{trekterug}, this is the restriction to $\g$ of 
the class 
\begin{equation*}\label{integrclass}
	\frac{k}{\mathrm{vol}_{\mu}(P)} [\si_{Q}]\in H^2(C_0^{\infty}(P)^{E}, \R)\,, 
\end{equation*}
where $\sigma_{Q}$ is the the cocycle
$\si_Q(f,g) =  \int_Q gdf\wedge(d\th)^{n-1}/(n-1)!$ of  \eqref{sg2}.
\end{proof}

\section{The quantomorphism group}

Let $P \rightarrow M$ be a prequantum bundle over a 
compact symplectic manifold $(M,\omega)$, and let $\th \in \Omega^1(P)$
be a connection 1-form with curvature $\omega$.
We apply Theorem~\ref{centralresult} to the 
identity component $G = \Quant(P,\th)_0$
of the quantomorphism group.
As its Lie algebra is isomorphic to the Poisson Lie algebra $\g = C^\oo(M)$,
an explicit description of the second Lie algebra cohomology 
is available \cite{JV}.
With the above construction,
we obtain a lattice of integrable classes in
$H^2(C^{\infty}(M),\R)$.

\subsection{Cohomology of the Poisson Lie algebra}

We describe the second Lie algebra cohomology $H^2(C^{\infty}(M), \R)$ 
in two different ways:
using \emph{Roger cocycles} related to $H^1(M,\R)$,
and 
using \emph{singular cocycles} related to $H_{2n-1}(M,\R)$.
The two pictures are linked by Poincar\'e duality.

\begin{Definition}
The \emph{Roger cocycle} \cite[\S 9]{R95} associated to a closed 1-form $\alpha$
on the $2n$-dimensional symplectic manifold $(M,\om)$
is defined by  
\begin{equation}
\label{rog}
\psi_\al(f,g):=\int_M f\al(X_g)\om^n/n!
=-\int_M\al\wedge fdg\wedge \om^{n-1}/(n-1)!.
\end{equation}
It was first defined for surfaces in \cite{Ki90}. 
\end{Definition}

The Roger cocycles link the first de Rham cohomology 
of $M$ to the second Lie algebra cohomology of the Poisson Lie algebra $C^{\infty}(M)$.

\begin{Theorem}\label{LAcohomology}{\rm\cite[\S 9]{R95}\cite[\S4]{JV}}
The Roger cocycles $\psi_{\alpha}$ and $\psi_{\alpha'}$ are cohomologous 
if and only if $\alpha - \alpha'$ is exact, and 
the corresponding map $[\alpha] \mapsto [\psi_{\alpha}]$ is an isomorphism
\[
	H^1(M,\R) \stackrel{\!\sim\,}{\rightarrow} H^2(C^{\infty}(M), \R)\,.
\]
\end{Theorem}

This shows (\cf \cite[\S2]{JV}) that every (locally convex) central extension 
$\R \rightarrow \widehat{\g} \rightarrow C^{\infty}(M)$
corresponds to a Roger cocycle \eqref{rog} with respect to some linear splitting
$C^{\infty}(M) \rightarrow \widehat{\g}$.
However, the cocycles that come from the Hamiltonian action 
of $\Quant(P,\th)$ on $\Gr_{2n-1}(P)$, using splittings of type \eqref{splitting},  are 
more closely related to \emph{singular} homology.  

\begin{Definition}\label{defpsiC}
The \emph{singular cocycle} $\psi_{C}$ on $C^{\infty}(M)$, associated to a singular $(2n-1)$-cycle $C$ on $M$, is defined by
\begin{equation}\label{sg22}
\psi_C(f,g):=
  \int_C gdf\wedge\om^{n-1}/(n-1)!\,.
\end{equation}
\end{Definition}
The Lie algebra 2-cocycles $\psi_{C}$ and $\psi_{C' }$ are cohomologous 
if and only if $C - C'$ is a boundary. 
\begin{Proposition}\label{cohthm}
The singular cocycle $\psi_{C}$ is cohomologous to the Roger cocycle $\psi_{\alpha}$
if and only if $[C] \in H_{2n-1}(M,\R)$ is Poincar\'e dual to $ [\alpha]\in H^1(M,\R)$. 
In particular, the map 
$[C] \mapsto [\psi_{C}]$ is an isomorphism
\begin{equation}\label{FullLACohomology}
H_{2n-1}(M,\R) \stackrel{\!\sim\,}{\rightarrow} H^2(C^{\infty}(M),\R)\,.
\end{equation}
\end{Proposition}
\begin{proof}
In view of the fact that $C^{\infty}(M) \simeq \R \oplus \X_{\ham}(M,\om)$ for compact $M$,
this follows from the discussion at the end of \cite[\S 5]{JV}.
\end{proof}

\subsection{An integrable lattice in $H^2(C^{\infty}(M),\R)$}

By applying Theorem \ref{centralresult} to the regular contact manifold $(P,\th)$
we obtain a lattice of integrable classes in 
the Lie algebra cohomology $H^2(C^{\infty}(M),\R)$.
In the following, we denote the lattice of integral classes by 
$H_{*}(M,\R)_{\Z}$ (homology) or $H^*(M,\R)_{\Z}$ (cohomology).

\begin{Corollary}[{\bf Singular version}]\label{singversion}
Let $[C] \in H_{2n-1}(M,\R)_{\Z}$ be in the image under $\pi_{*}$
of $H_{2n-1}(P,\R)_{\Z}$. Then
the Lie algebra extension corresponding to the class
\[
	\frac{n+1}{2\pi\, \mathrm{vol}(M)} [\psi_{C}]
\] 
integrates to a central 
extension 
of $\Quant(P,\th)_0$ by~$S^1$. 
In the above expression,  $\mathrm{vol}(M) = \int_{M}\omega^n/n!$ is the Liouville volume of $M$, 
and $[\psi_{C}] \in H^2(C^{\infty}(M), \R)$ is the singular class \eqref{sg22}.
\end{Corollary}
\begin{proof}
We apply Theorem~\ref{centralresult} to the unit component  
$G = \Diff(P,\th)_0$ of the quantomorphism group.
By \cite[VIII.4]{Om74}, this 
is a Fr\'echet Lie group, with Lie algebra $\g$ isomorphic to the 
Poisson Lie algebra $C^{\oo}(M)$.
By Proposition~\ref{LAsame}, the group $\Diff(P,\th)_0$ 
coincides with the group $\Diff_{\ex}(P,\th)$ of exact 
strict contactomorphisms. 

Recall that $(P,\th)$ is a regular contact manifold,
for which
$d\th = \pi^*\omega$ and $C^{\infty}(P)^{E} = \{\pi^*f \,;\, f \in C^{\infty}(M)\}$.
If $Q \subseteq P$ is an embedded, closed, oriented submanifold, then 
the Hamiltonian action of $\Quant(P,\th)_{0}$ on the connected component $\M \subseteq \Gr_{2n-1}(P)$
of $Q$ gives rise 
to a central Lie group extension with Lie algebra cocycle on $C^\oo(M)$
\[
\psi(f,g) = \frac{1}{\mathrm{vol}_{\mu}(P)}\si_Q(\pi^*f,\pi^*g)= \frac{1}{\mathrm{vol}_{\mu}(P)} \int_{Q} \pi^* \left(gdf\wedge\omega^{n-1}\right)/(n-1)!\,.
\]
Expressing the volume of $P$ as
$\mathrm{vol}_{\mu}(P)=\frac{2\pi}{n+1}\mathrm{vol}(M)$, we find
\[
\psi(f,g) = \frac{n+1}{2\pi \,\mathrm{vol}(M)} \,\psi_{\pi_*Q}(f,g)\,,\\
\]
where  $\pi_* Q$
is the pushforward along $\pi \colon P \rightarrow M$ of the singular $(2n-1)$-cycle represented by 
the embedded closed submanifold $Q \subseteq P$, and $\psi_{\pi_*Q}$
is the singular cocycle of \eqref{sg22}.
By Remark \ref{pd}, every class in $H^2(P,\R)_{\Z}$ can be represented
by an oriented, embedded submanifold $Q$, so with $[C] = \pi_* [Q]$ in 
$H_{2n-1}(M,\R)_{\Z}$, the result follows.
\end{proof} 

\begin{Remark}[\bf Triviality of Lie algebra extensions]\label{Rk:Triviality}
Note that from the above proof,
it follows that the Lie algebra extension corresponding to the Hamiltonian action of 
$\Quant(P,\th)_{0}$ on the connected component $\M$ of $Q$ in $ \Gr_{2n-1}(P)$ is trivial if and only if 
$[\psi_{\pi_{*}Q}] \in H^2(C^{\infty}(M),\R)$ is zero. By Proposition \ref{cohthm}, this is the case
if and only if
$\pi_*[Q] \in H_{2n-1}(M,\R)$ vanishes. 
\end{Remark}

Using Poincar\'e duality, we translate this to Roger cocycles and de Rham cohomology.
For a smooth map $f \colon M \rightarrow N$, we denote by $f_{!} \colon H^*(M,\Z) \rightarrow H^*(N,\Z)$
the map that corresponds to $f_* \colon H_*(M,\Z) \rightarrow H_*(N,\Z)$ under Poincar\'e duality.
For the prequantum bundle $\pi \colon P \rightarrow M$, the induced map
\[\pi_{!} \colon H^{k}(P,\R)_{\Z} \rightarrow H^{k-1}(M,\R)_{\Z}\] on 
integral classes in de Rham cohomology is fiber integration.

\begin{Corollary}[{\bf de Rham version}]\label{derhamversion}
For every class $[\alpha]$ in the sublattice $\pi_{!} (H^{2}(P,\R)_{\Z})$
of $H^{1}(M,\R)$,
the Lie algebra extension corresponding to the class
\[
	\frac{n+1}{2\pi\, \mathrm{vol}(M)} [\psi_{\alpha}]
\]
in $H^2(C^{\infty}(M), \R)$ integrates to a central 
extension 
of $\Quant(P,\th)_0$ by~$S^1$.
\end{Corollary}
\begin{proof}
The Roger cocycle 
$\psi_{\alpha}$ is cohomologous to the singular cocycle $\psi_{C}$
if $[C] \in H_{2n-1}(M,\R)_{\Z}$ is Poincar\'e dual to
$[\alpha] \in H^1(M,\R)_{\Z}$.
If $[C] = \pi_* [Q]$ for $[Q] \in H_{2n-1}(Q,\R)_{\Z}$, then 
$[\alpha] = \pi_{!} [\eta_{Q}]$ for the Poincar\'e dual $[\eta_{Q}]$ of $Q$, as
\[\int_{M} \pi_{!} \eta_{Q} \wedge \gamma = \int_{P}\eta_{Q} \wedge \pi^*\gamma = \int_{Q} \pi^*\gamma = \int_{C}\gamma\,.\]
The result now follows from Corollary \ref{singversion}.
\end{proof}

The lattice $\pi_{!}(H^2(P,\R)_{\Z})$, 
which yields the integrable classes in Lie algebra cohomology,
is contained in the lattice $(\pi_{!}H^2(P,\R))_{\Z}$ of integral classes 
in $\pi_{!}H^2(P,\R)$. 
Note however that it can be strictly smaller, cf.\  \S \ref{sec:thurston}.
The following proposition
is helpful in determining this lattice.

\begin{Proposition}\label{Gysin}
The image of $\pi_{!} \colon H^{2}(P,\Z) \rightarrow H^{2}(M,\Z)$
is the kernel of taking the cup product with the Euler class $[P]\in H^2(M,\Z)$ 
of the bundle $P$,
\begin{equation}\label{imageofpr}
\pi_{!}H^{2}(P,\Z) = \{[\alpha] \in H^1(M,\Z)\,;\, [P] \smallsmile [\alpha] = 0\}.
\end{equation}
\end{Proposition}
\noindent Note that the image of $[P]$ in $H^2(M,\R)$ 
is the class $[\omega]$ of the symplectic form.
\begin{proof}
This follows immediately from
the Gysin long exact sequence 
in integral cohomology, associated to the principal $S^1$-bundle $P\to M$, 
\[
\dots \to 
H^2(P,\Z)\stackrel{\pi_!}{\to}
H^1(M,\Z)\stackrel{[P]\smallsmile \,\cdot\,}{\to}H^3(M,\Z)\stackrel{\pi^*}{\to}H^3(P,\Z)\to\dots,
\]
see e.g.\ \cite[\S VI.13]{Br93}.
\end{proof}

Similarly, it follows from Poincar\'e duality (or the Gysin sequence in homology, \cite[\S 9.3]{Sp66}), that
\begin{equation}\label{GysinHomology}
\pi_{*} (H_{2n-1}(P,\Z)) = \{[C] \in H_{2n-1}(M,\Z)\,;\, [C_{P}] \smallfrown [C] = 0\}\,,
\end{equation}
where $[C_{P}]\in H_{2n-2}(M,\Z)$, Poincar\'e dual to $[P]$, is the zero set of a transversal
section of the prequantum line bundle 
$P \times_{S^1}\mathbb{C} \rightarrow M$.

\subsection{Examples}
We calculate the integrable classes in $H^2(C^{\infty}(M),\R)$ that correspond to our
group extensions for a number of explicit examples. 
For compact surfaces, they span the second Lie algebra cohomology, 
whereas for compact K\"ahler manifolds of $\mathrm{dim}_{\R}\geq 4$, 
they are all trivial. 
For non-K\"ahler symplectic manifolds, our method yields non-trivial integrable
classes, but they do not necessarily span $H^2(C^\oo(M),\R)$.
We illustrate this at the hand of Thurston's nilmanifold,
which was historically the first example of a non-K\"ahler symplectic manifold.

\subsubsection{Compact surfaces}
Let $M$ be a compact orientable 2-dimensional manifold of genus $g$, with generators $[a_{1}], \ldots, [a_{g}]$ and 
$[b_{1}], \ldots, [b_{g}]$ of
$H_{1}(M,\Z)$. 
A symplectic form $\omega \in \Omega^2(M)$ is prequantizable
if and only if $\mathrm{vol}(M) \in \Z$.
As $H_3(M,\Z) = \{0\}$, equation \eqref{GysinHomology} shows that
$\pi_* H_1(P,\Z) = H_{1}(M,\Z)$.
From Corollary \ref{singversion}, we thus obtain:
\begin{Corollary}
For $k_i, l_i \in \Z$, the Lie algebra cocycles
on the Poisson Lie algebra $C^\oo(M)$
\begin{equation} \label{2dcase}
\psi(f,g) =\frac{1}{\pi \mathrm{vol}(M)} \left(\sum_{i=1}^{g} k_i \int_{a_{i}}  gdf + 
\sum_{i=1}^{g} l_i \int_{b_{i}} gdf\right)
\end{equation}
integrate to central $S^1$-extensions of the group $\Quant(P,\th)_{0}$
of quantomorphisms.
\end{Corollary}
By Theorem~\ref{LAcohomology}, the $\R$-span of this integrable lattice 
is the full second Lie algebra cohomology of the Poisson Lie algebra $C^{\infty}(M)$.
{The above result on integrable cocycles appears to be new.}

\subsubsection{K\"ahler manifolds}

If $M$ is a prequantizable compact K\"ahler manifold 
of dimension $2n$, $n\ge 2$, then the map 
\[H^1(M,\R) \rightarrow H^{2n-1}(M,\R) \,;\quad [\alpha] \mapsto [\omega]^{n-1}\wedge [\alpha]\]
is an isomorphism by the Hard Lefschetz Theorem.
Since $n\ge 2$, the map $[\alpha] \mapsto [\omega] \wedge [\alpha]$ is injective, 
so Proposition \ref{Gysin} implies that $\pi_{!}H^{2}(P,\R)_{\Z} = \{0\}$.
From Remark \ref{Rk:Triviality}, we then obtain the following result.
\begin{Corollary}
If $M$ is a compact K\"ahler manifold of real dimension $\geq 4$, 
then the central Lie group extension derived from the Hamiltonian action of 
$\Quant(P,\th)_{0}$ on
the connected component $\M$ of  $\mathrm{Gr}_{2n-1}(P)$ splits at the Lie algebra level.
\end{Corollary}

In particular, the Hamiltonian action of $\Quant(P,\th)_{0}$ on $\M$
lifts to 
an action of the universal cover $\widetilde{\Quant}(P,\th)_{0}$ on 
the prequantum bundle $\cP \rightarrow \M$. 
For compact K\"ahler manifolds of $\mathrm{dim}_{\R} \geq 4$, we thus
obtain a \emph{linear} representation of 
$\widetilde{\Quant}(P,\th)_{0}$ on the space 
of sections of the prequantum line bundle $\mathcal{L} \rightarrow \M$
associated to $\cP$. This marks a qualitative difference with
the case $\mathrm{dim}_{\R} = 2$, where
the central Lie algebra deformation \eqref{2dcase} occurs. 

\subsubsection{Thurston's nilmanifold}\label{sec:thurston}
A \emph{nilmanifold} $M = \Gamma \backslash N$ is a compact homogeneous space for a 
connected nilpotent Lie group $N$. 
Without loss of generality, one may assume that $N$ is 1-connected, and
$\Gamma \subseteq N$ discrete and co-compact \cite{Ma49}.
If $\n$ is the Lie algebra of $N$, then 
by \cite{No54},
the inclusion 
$\bigwedge \n^* \hookrightarrow \Omega(M)$ as left invariant forms
yields an isomorphism 
between the Lie algebra cohomology $H^*(\n,\R)$ of $\n$ and the de Rham cohomology 
$H^{*}(M,\R)$ of $M$.
This remains true over rings of integers localized at small primes \cite{LP82}.

To illustrate that nontrivial lattices of integrable cocycles for $C^{\infty}(M)$ exist 
in dimension $\geq 4$, we consider
the quotient $M_{r} = \Gamma \backslash N$ with $N = \mathrm{Heis}(\R,r) \times \R$
and $\Gamma = \mathrm{Heis}(\Z,r) \times \Z$, where $\mathrm{Heis}(R,r)$ is 
the \emph{Heisenberg group} over the ring $R$ at level $r\in \N$, 
\[
\mathrm{Heis}(R,r) := 
\left \{
\begin{pmatrix}
1 & u & h/r\\
0 & 1 & v\\
0 & 0 & 1
\end{pmatrix}
 \,;\, u, v, h \in R
\right\}\,.
\]
For $r=1$, this is Thurston's symplectic manifold $M^4$ \cite{Th76}. We 
include the case $r\neq 1$ to illustrate the role that the torsion of $H^2(M,\Z)$
plays in determining the lattice of integrable cocycles.

The Lie algebra $\n$ is generated by
$x,p,h$ and $z$, with $h,z$ central and $[x, p] = r h$.
The left invariant forms corresponding to the dual basis are
$x^* = du$, $p^* = dv$, $z^* = dz$ and $h^* = dh + r \,udv$.
The differential $\delta \colon \bigwedge\n^* \rightarrow \bigwedge \n^*$
in  \eqref{CE}
is determined by its action on generators:  $\delta h^* = r\, x^* \wedge p^*$ and
$\delta x^* = \delta p^* = \delta z^* = 0$. 

The cohomology of the Eilenberg-MacLane space $M_{r} \simeq K(\Gamma, 1)$
is readily calculated from $H^*(B\Gamma,\Z)$. 
From \cite[\S 6.1]{LP82} and the K\"unneth formula, one finds
\[
H^*(M_{r},\Z) \simeq H^*(\n,\Z) = \bigwedge \left [x^*,p^*, z^*, x^*\wedge h^*, p^*\wedge h^* \right ] /(r x^*\wedge p^*)\,.
\]
For $a,b \in \Z - \{0\}$, we define the (integral) symplectic form $\omega_{ab} \in \Omega^2(M_{r})$ by
\[\omega_{ab} = a h^* \wedge x^* +  b z^* \wedge p^* = a (dh\wedge du - r u du \wedge dv) + b\, dz \wedge dv\,.\] 
It determines the Euler class $[P_{abc}] \in H^2(M_{r},\Z)$ of the prequantum line bundle only up to torsion;
\[ [P_{abc}] = a h^* \wedge x^* +  b z^* \wedge p^* + c x^* \wedge p^*\,, \]
where $c \in \{ 0, \ldots, r-1\}$ labels the different prequantum line bundles with 
the same curvature class $[\omega_{ab}] \in H^2(M_{r}, \R)$.
The kernel of the cup product with the Euler class is given by
\[
\mathrm{Ker}
\Big(
[P_{abc}] \smallsmile \,\cdot\, \colon H^1(M_{r},\Z) \rightarrow H^3(M_{r},\Z)
\Big) = \pi_{!} H^2(P_{abc},\Z) =
\{tx^* \,;\, r|tb\}\,,
\]
as $[P_{abc}] \smallsmile t x^* = tb z^* \wedge p^* \wedge x^*$ is a multiple of 
$\delta (z^* \wedge h^*)$ if and only if $r | tb$.

From Corollary \ref{derhamversion} with $n=2$ and $\mathrm{vol}(M_{r}) = ab$, we then obtain the following lattice of integrable 
classes in second 
Lie algebra cohomology:
\begin{Corollary}
For $(M_{r}, \omega_{ab})$ and $k\in \Z$, the 2-cocycles 
\[
\psi(f,g) = \frac{3 r  k}{2\pi a \,\,\mathrm{gcd}(r,b)} \int_{M} f dg \wedge du \wedge dv \wedge dz
\]
for the Poisson Lie algebra $C^{\infty}(M_{r})$ 
are integrable to the identity component of the quantomorphism group 
$\Quant(P_{abc},\th)_{0}$.
\end{Corollary}
Since $H^2(C^{\infty}(M_r),\R) \simeq H^1(M_r,\R)$ 
by Theorem \ref{LAcohomology}, we find a single ray spanned by 
integrable classes in this 3-dimensional cohomology space.


\section{The Hamiltonian group}

In this final section, we briefly describe how 
the central extensions of the quantomorphism group
$\Quant(P,\th)_{0}$, obtained in Corollary \ref{derhamversion},
can be pulled back by a homomorphism 
$\widetilde{\Ham}(M, \om) \rightarrow \Quant(P,\th)_0$. This
yields central $S^1$-extensions 
of the universal covering group $\widetilde{\Ham}(M,\om)$.

The homomorphism $\widetilde{\Ham}(M, \om) \rightarrow \Quant(P,\th)_0$
is obtained as follows. 
Since $M$ is compact, the Kostant-Souriau extension 
\eqref{poissonrijtje}
is split (\cite[Corollary 3.5]{JV})
by the Lie algebra homomorphisms
\[
\R \stackrel{\rho}{\leftrightarrows} C^{\infty}(M) \stackrel{\kappa}{\leftrightarrows} \X_{\ham}(M),
\]
%
defined by
\[
\rho(f) := \frac{1}{\mathrm{vol}(M)}\int_{M} f \omega^n/n!\,, 
\quad
\kappa(X_{f}) := f - \rho(f)\,.
\]
By Lie's Second Theorem for regular Lie groups
\cite[Thm.\ 40.3]{KM97}, 
these Lie algebra homomorphisms integrate to group homomorphisms
\[
\R \stackrel{R}{\leftrightarrows} \widetilde{\Quant}(P,\th)_{0} \stackrel{K}{\leftrightarrows} \widetilde{\Ham}(M,\om),
\]
%
on the universal covering groups. 
This yields the following commutative diagram:
\begin{center}
\begin{tikzpicture}
  \matrix (m) [matrix of math nodes,row sep=1.5em,column sep=2em,minimum width=2em]
  {
      & & \widehat{\Quant}(P,\th)_{0}\\[-4pt]
     \R & \widetilde{\Quant}(P,\th)_{0} & \Quant(P,\th)_{0} \\
     & \widetilde{\Ham}(M,\om) & \Ham(M,\om). \\};
  \path[-stealth]
    (m-1-3) edge (m-2-3)
    (m-2-3) edge (m-3-3)
    (m-2-1) edge (m-2-2)
    (m-2-2) edge [bend right, dashed] node[above]{\scriptsize $R$} (m-2-1)
    (m-2-2) edge node [above] {{\scriptsize Pr}} (m-2-3)
    (m-3-2) edge (m-3-3)
    (m-3-2) edge [dashed] (m-2-3)
    (m-2-2) edge (m-3-2)
    (m-3-2) edge [dashed, bend right] node[right]{{\scriptsize $K$}} (m-2-2);
\end{tikzpicture}
\end{center}
If we pull back the central $S^1$-extension 
$\widehat{\Quant}(P,\th)_{0} \rightarrow \Quant(P,\th)_{0}$
along the homomorphism
$\mathrm{Pr}\o K \colon \widetilde{\Ham}(M, \om) \rightarrow \Quant(P,\th)_0$,
we obtain a central Lie group extension 
$\widehat H \rightarrow \widetilde{\Ham}(M, \om)$
by $S^1$,
%
\begin{center}
\begin{tikzpicture}
  \matrix (m) [matrix of math nodes,row sep=1.5em,column sep=3em,minimum width=2em]
  {
      \widehat H& \widehat{\Quant}(P,\th)_{0}\\
     \widetilde{\Ham}(M,\om)_{0} & \Quant(P,\th)_{0}. \\};
  \path[-stealth]
    (m-1-1) edge [dashed] (m-1-2)
    (m-1-1) edge [dashed](m-2-1)
    (m-1-2) edge (m-2-2)
    (m-2-1) edge node [above] {{\scriptsize $\mathrm{Pr}\circ K$}} (m-2-2);
\end{tikzpicture}
\end{center}
If $\psi$ is the cocycle of $C^{\infty}(M)$ corresponding to the central extension
$\widehat{\Quant}(P,\th)_0$ of $\Quant(P,\th)_{0}$, then
$\kappa^*\psi$ is
the corresponding Lie algebra cocycle on $\X_{\ham}(M)$.
The pullback by $\ka$ of the Roger cocycle \eqref{rog}
is $$(\ka^*\psi_\al)(X_f,X_g)=\int_M f\al(X_g)\om^n/n!\,,$$
and {the pullback} of the singular cocycle \eqref{sg22} is
$$(\ka^*\psi_C)(X_f,X_g)= 
\int_C gdf \wedge\omega^{n-1}/(n-1)!\,.
$$
Both expressions are independent of
the choice of Hamiltonian functions.

\begin{Proposition}\label{intham}
If the  2-cocycle $\psi$ on $C^\oo(M)$ can be integrated to a central extension of the group of quantomorphisms $\Quant(P,\th)_0$,
then $\kappa^*\ps$ 
can be integrated to a central extension of the universal covering group
$\widetilde\Ham(M,\om)$ of the group of Hamiltonian diffeomorphisms. 
\end{Proposition}

\section*{Acknowledgements}
We thank Stefan Haller for valuable comments and remarks, 
which improved the final version of the paper.
C.V.\ was supported by the grant PN-II-ID-PCE-2011-3-0921
of the Romanian National Authority for Scientific Research.
B.J.\ was supported by the NWO grant 
613.001.214 ``Generalised Lie algebra sheaves".


\end{document}